\documentclass[12pt]{article}
\usepackage[a4paper, margin=1.1in]{geometry}
\usepackage{amsmath,amsthm,mathtools}
\usepackage{amsfonts,amssymb,ragged2e}
\usepackage[raggedright]{titlesec}
\usepackage{graphicx}
\usepackage{float}
\usepackage{caption}
\usepackage{chngcntr}
\usepackage{fncylab}
\usepackage{verbatim}
\usepackage{cite}
 \theoremstyle{plain} 
\newtheorem{thm}{Theorem}[section]
\newtheorem{cor}[thm]{Corollary}

\newtheorem{prop}{Proposition}[section]

\newtheorem{defn}{Definition}[section]
\newtheorem{op}{Operation}[section]
\theoremstyle{remark}
\newtheorem{rem}{Remark}[section]

\begin{document}

\title
{\bf{Energy and Randi{\'c} energy of special graphs }}
\author {\small Jahfar T K \footnote{jahfartk@gmail.com} and Chithra A V \footnote{chithra@nitc.ac.in} \\ \small Department of Mathematics, National Institute of Technology, Calicut, Kerala, India-673601}
\date{ }
\maketitle
\begin{abstract}
 In this paper, we determine the Randi{\'c} energy of the $m$-splitting graph, the $m$-shadow graph and the $m$-duplicate graph of a given graph, $m$ being an arbitrary integer. Our results allow the construction of an infinite sequence of graphs having the same Randi{\'c} energy. Further, we determine some graph invariants like the degree Kirchhoff index, the Kemeny's constant and the number of spanning trees of some special graphs. From our results, we indicate how to obtain infinitely many pairs of equienergetic graphs, Randi{\'c} equienergetic graphs and also,  infinite families of integral graphs.
\end{abstract}

\hspace{-0.6cm}\textbf{AMS classification}: 05C50
\section{Introduction}
In this paper, we consider simple connected graphs. Let $G=(V,E)$ be a simple graph of order $p$ and size $q$ with vertex set $V(G)=\{v_{1}, v_{2},...,v_{p}\}$ and edge set $E(G)=\{e_{1}, e_{2},...,e_{q}\}$. The degree of a vertex $v_i$ in $G$ is the number of edges incident to it and is denoted by $d_i=d_G(v_i)$. The adjacency matrix $A(G)=[a_{ij}]$ of the graph $G$ is a square symmetric matrix of order $p$ whose $(i,j)^{th}$ entry is defined by

 \[a_{i,j}=\begin{matrix}
 \begin{cases}
 1, & \text{if $v_i$ and $v_j$ are adjacent, }\\    
 0, & \text{otherwise. }
 \end{cases}
 \end{matrix}.\] 
  The eigenvalues $ \lambda_{1},\lambda_{2},...,\lambda_{p}$ of the graph $G$ are defined as the eigenvalues of its adjacency matrix $A(G)$.  If $ \lambda_{1},\lambda_{2},...,\lambda_{t}$ are the distinct eigenvalues of $G$, the spectrum of $G$ can be written as\\ $Spec(G)=\begin{pmatrix}
\lambda_1&\lambda_2& ...&\lambda_t\\

m_1&m_2&...&m_t\\
\end{pmatrix}$, where $m_j$ indicates the algebraic multiplicity of the eigenvalue $\lambda_j$, $1\leq j\leq t$ of $G$.
The energy \cite{gutman1978energy} of the graph $G$ is defined as  $\varepsilon(G)=\displaystyle\sum_{i=1}^{p} |{\lambda_i}|$. More results on graph energy are reported in \cite{gutman1978energy,balakrishnan2004energy}. The Randi{\'c} matrix $R(G)=[R_{i,j}]$ of a graph $G$ is a square matrix of order $p$ whose  $(i,j)^{th}$ entry is\[R_{i,j}=\begin{matrix}
\begin{cases}
\frac{1}{\sqrt{d_id_j}}, & \text{if $v_i$ and $v_j$ are adjacent, }\\    
0, & \text{otherwise. }
\end{cases}
\end{matrix}.\]  The  eigenvalues of $R(G)$ are called Randi{\'c} eigenvalues of $G$ and it is denoted by
$\rho_{i},1\leq i \leq p$.  If $ \rho_{1},\rho_{2},...,\rho_{s}$ are the distinct Randi{\'c} eigenvalues of $G$, then the Randi{\'c} spectrum of $G$ can be written as $RS(G)=\begin{pmatrix}
\rho_1&\rho_2& ...&\rho_s\\

m_1&m_2&...&m_s\\
\end{pmatrix}$, where $m_j$ indicates the algebraic multiplicity of the eigenvalue $\rho_j$,$1\leq j\leq s$ of $G$. If $G$ has no isolated vertices, then $R(G)=D^{-1/2}A(G)D^{-1/2}$, where $D^{1/2}$  is the diagonal matrix with diagonal entries $\frac{1}{\sqrt{d_i}}$ for every $i$, $1\leq i \leq p$  \cite{bozkurt2010randic}. Randi{\'c} energy of $G$ is defined as $\varepsilon_R(G)= \displaystyle\sum_{i=1}^{p} |{\rho_i}|$\cite{bozkurt2010randic,gutman2014randic}. More results on Randi{\'c} energy are reported in \cite{das2017normalized,alikhani2015randic}.  A graph is said to be integral if all the eigenvalues of its adjacency matrix are integers. 
Two non-isomorphic graphs $G_1$ and $G_2$ of the same order are said to be equienergetic if $\varepsilon (G_1)=\varepsilon(G_2)$ \cite{ramane2004equienergetic}. In analogy to this, two graphs $G_1$ and $G_2$ of same order are said to be Randi{\'c} equienergetic if $\varepsilon_R(G_1)=\varepsilon_R(G_2)$\cite{altindag}. The normalized Laplacian matrix $\mathcal{L}(G)=(\mathcal{L}_{ij})$ is the square matrix of order $p$ whose $(i,j)^{th}$ entry is defined as,
\[\mathcal{L}_{ij}=\begin{cases}
1, \text{if $v_i=v_j$ and $d_i\ne 0$,}\\
-\frac{1}{\sqrt{d_id_j}},\text{if $v_i$ and $v_j$ are adjacent in $G$,}\\
0,  \text{otherwise}.
\end{cases}.\] For a graph $G$ without isolated vertices, the normalized Laplacian matrix can be written as   $$ \mathcal{L}(G)=I_n- D^{-\frac{1}{2}}(G)A(G)D^{-\frac{1}{2}}(G).$$ 
The eigenvalues of the matrix $\mathcal{L}(G)$ are called the normalized Laplacian eigenvalues of $G$ and it is denoted by $0=\tilde{\mu}_1(G)\leq\tilde{\mu}_2(G)...\leq\tilde{\mu}_p(G)$. Let $G$ be a graph without isolated vertices. Then its normalized Laplacian matrix $\mathcal{L}(G)$ and Randi{\'c} matrix $R(G)$  are related by $\mathcal{L}(G)=I-R(G)$. The normalized Laplacian eigenvalues $\tilde{\mu}_i(G)$ and Randi{\'c} eigenvalues $\rho_i(G)$ are related by $\tilde{\mu}_i(G)=1-\rho_i(G),$ for $i=1,2,...,p$. The degree  Kirchhoff index of connected graph $G$ is defined in \cite{chen2007resistance} as $$Kf^*(G)=\sum_{i<j}^{}d_ir_{i,j}^*d_j$$ where $r_{i,j}^*$ denotes the resistance distance \cite{klein1993resistance} between vertices $v_i$ and $v_j$ in a graph $G$.   In \cite{chen2007resistance}, the authors proved that $$ Kf^*(G)=2q\sum_{i=2}^{p}\frac{1}{\tilde{\mu}_i(G)}.$$
 In general, the computation of the  degree Kirchhoff index of a graph is a difficult thing. Here we obtained the formula for finding the degree Kirchhoff index of some families of graphs. 
 The Kemeny's constant $K(G)$ of a connected graph $G$ \cite{butler2016algebraic} is defined in terms of normalized Laplacian as $$K(G)=\sum_{i=2}^{p}\frac{1}{\tilde{\mu}_i(G)}.$$
The various  applications of the Kemeny's constant
to perturbed Markov chains, random walks on directed graphs are studied in \cite{hunter2014role}.
 The number of spanning trees (distinct spanning subgraphs of $G$ that are trees) of  $G$  \cite{chung58183spectral} can be expressed in terms of the normalized Laplacian eigenvalues as  $$t(G)=\frac{\prod_{i=1}^{p}d_i\prod_{i=2}^{p}\tilde{\mu}_i(G)}{\sum_{i=1}^{p}d_i}.$$
\\We use the notations  $K_n,C_n$ and $K_{1,n-1}$ throughout this paper to denote the complete graph, the cycle and the star graph on $n$ vertices respectively. Let $J_m$ be the $m\times m$ matrix of all ones and $I_m$ be the identity matrix of order $m$.\\ 
 \indent The rest of the paper is organized as follows. In Section 2, we give a list of some previously known results which are useful for further reference in this paper. In Section 3, Randi{\'c} energy of the $m$-splitting graph, the $m$-shadow graph and the $m$-duplicate graphs are obtained. In Section 4, our results allow the construction of  an infinitely many integral and Randi{\'c}  integral graphs. Also, our results show how to construct equienergetic and Randi{\'c} equienergetic graphs. In Section 5, we discuss the graph invariants like the degree Kirchhoff index, the Kemeny's constant and the number of spanning trees of resulting graphs from various graph operations. 
	 
 \section{Preliminaries}
In this section, we recall the concepts of the $m$-splitting graph, the $m$- shadow graph and the $m$-duplicate graph of a graph and list some results that will be used in the subsequent sections.
\begin{defn}\textnormal{\cite{cvetkovic1980spectra}}
The Kronecker product of two graphs  $G_1$ and $G_2$ is the graph $ G_1\times G_2$ with vertex set $V(G_1) \times V(G_2)$ and the vertices $(x_1,x_2)$ and $(y_1,y_2)$ are adjacent if and only if $(x_1,y_1)$ and $(x_2,y_2)$ are edges in  $G_1$ and $G_2$ respectively.
\end{defn}
\begin{defn}\textnormal{\cite{cvetkovic1980spectra}}
Let $A\in R^{ m \times n} $,$B\in R^{ p \times q}. $ Then the Kronecker product of $A$ and $B$ is defined as follows \[ A\otimes B=\begin{bmatrix}
a_{11}B & a_{12}B & a_{13}B & \dots  & a_{1n}B \\
a_{21}B & a_{22}B & a_{23}B & \dots  & a_{2n}B \\
\vdots & \vdots & \vdots & \ddots & \vdots \\
a_{m1}B & a_{m2}B & a_{m3}B & \dots  & a_{mn}B\\
\end{bmatrix}.\]
\end{defn}
\begin{prop}\textnormal{\cite{cvetkovic1980spectra}}
Let $A,B\in R^{n\times n}$. Let $\lambda$ be an eigenvalue of matrix $A$ with corresponding eigenvector $x$ and	$\mu$ be an eigenvalue of matrix $B$ with corresponding eigenvector $y$, then $\lambda\mu$ is an eigenvalue of $A\otimes B$ with corresponding eigenvector $x\otimes y.$
\end{prop}	
\begin{defn}\textnormal{\cite{vaidya2017energy}}
	Let $G$ be a simple $(p,q)$ graph. Then the $m$-splitting graph of a graph $G$, $Spl_m(G)$ is obtained by adding to each vertex $v$ of $ G $  new $ m $ vertices say, $v_1,v_2,...,v_m$ such that $v_i$ ,$1\leq i \leq m$
	is adjacent to each vertex that is adjacent to $v$ in G. The adjacency matrix of the $m$-Splitting graph of the graph $G$ is \[ A(Spl_m(G)) =\begin{bmatrix}
	A(G) & A(G) &  A(G) & \dots  & A(G)\\
	A(G) & O & O & \dots  & O \\
	\vdots & \vdots & \vdots & \ddots & \vdots \\
	A(G) & O & O & \dots  & O\\
	\end{bmatrix}_{(m+1)p}.\] If $m=1$, the $m$-Splitting graph of the graph $G$ is known as splitting graph of $G$\textnormal{\cite{sampathkumar1980splitting}}, denoted by $Spl(G)$. The number of vertices and the number of edges in $Spl_m(G)$ are $(m+1)p$ and  $(m+1)q$ respectively.
   \end{defn}
	\begin{prop}\textnormal{\cite{vaidya2017energy}}
		The energy of the $m$-splitting graph  of $G$ is $\varepsilon (Spl_m(G))=\sqrt{1+4m}\varepsilon(G).$
	\end{prop} 
\begin{defn}\textnormal{\cite{vaidya2017energy}}
	Let $G$ be a simple $(p,q)$ graph. Then the $m$-shadow graph $D_m(G)$ of a connected graph $G$ is constructed by taking m copies of $G$  say, $G_1,G_2,...,G_m$ then join each vertex $u$ in $G_i$ to the neighbors of the corresponding vertex $v$  in  $G_j,1\leq i\leq m,1\leq j\leq m.$
	The adjacency matrix of the $m$-shadow graph of $G$ is  
	\[ A(D_m(G)) =\begin{bmatrix}
	A(G) & A(G) &  A(G) & \dots  & A(G)\\
	A(G) & A(G) & A(G) & \dots  & A(G) \\
	\vdots & \vdots & \vdots & \ddots & \vdots \\
	A(G) & A(G) & A(G) & \dots  & A(G)\\
	\end{bmatrix}_{mp}.\]If $m=2$, the $m$-shadow graph of $G$ is known as shadow graph of $G$\textnormal{\cite{munarini2008double}}. The number of vertices and the number of edges in $D_m(G)$ are $pm$ and  $m^2q$ respectively.
	\begin{prop}\label{2.3}\textnormal{\cite{vaidya2017energy}}
	The energy of the $m$-shadow graph of $G$ is $\varepsilon(D_m(G))=m\varepsilon(G).$
	\end{prop}
	\end{defn}	
\begin{defn} \textnormal{\cite{sampathkumar1973duplicate}}
Let $G=(V,E)$ be a simple $(p,q)$ graph with vertex set $V$ and edge set $E$. Let $V'$ be a set such that $V\bigcap V'=\emptyset$, $|V|=|V'|$ and $f:V\rightarrow V'$ be bijective $($for  $a \in V$ we write $f(a)$ as $a'$ for convenience $)$. A duplicate graph of $G$ is $D(G)=(V_1,E_1)$, where the vertex set $V_1=V\cup V'$ and the edge set $E_1$ of $D(G)$ is defined as, the edge ab is in $E$ if and only if both $ab'$ and $a'b$ are in $E_1$.
 In general the $m$-duplicate graph of the graph $G$, $D^m(G)$ is defined as $D^m(G)=D^{m-1}(D(G))$. \\The number of vertices and the number of edges in the $m$-duplicate graph of the graph  are $2^mp$ and $2^mq$ respectively. With suitable labeling of the vertices, the adjacency matrix of $D(G)$ is \[A(D(G))=\begin{bmatrix}
O_{p\times p}&A(G)\\
A(G)&O_{p\times p}\\
\end{bmatrix}. 
\]
\begin{prop}\textnormal{\cite{indulal2006pair}}
	The energy of the duplicate graph of $G$ is $\varepsilon(D(G))=2\varepsilon(G).$
\end{prop}
 In \textnormal{\cite{patil2015tensor}}, authors remarked that,  the $m$-duplicate graph of $G$, $D^m(G)= G\times K_2\times K_2...\times K_2$ $(K_2$ repeats m-times$)$ . The energy of the $m$- duplicate graph of $G$ is  $\varepsilon(D^m(G))=\varepsilon(G).\varepsilon(K_2)...\varepsilon(K_2)=2^m\varepsilon(G)$\textnormal{\cite{balakrishnan2004energy}}. 

\end{defn} 

\section{Randi{\'c} energy of the $m$-splitting, the $m$-shadow and the $m$-duplicate graphs }
In this section, we present the Randi{\'c} energy of the $m$-splitting graph, the $m$-shadow graph and the $m$-duplicate graphs of $G$. Also, we obtain some new families of Randi{\'c} equienergetic graphs. In addition, our results show how to construct  infinitely many families of integral graphs.

\begin{thm}\label{3.1}
	Let $G$ be a simple $(p,q)$ graph without isolated vertices. Then the Randi{\'c} energy of the $m$-splitting graph of $G$ is $\varepsilon_R(Spl_m(G))=\frac{2m+1}{m+1}\varepsilon_R(G).$
\end{thm}
\begin{proof}
	The Randi{\'c} matrix of $m$- splitting graph of $G$ is $R(Spl_m(G))$\begin{align} &=\left[\begin{smallmatrix}
	((m+1)D)^{-\frac{1}{2}} & O &   \dots  &O\\
	O & D^{-\frac{1}{2}} &  \dots  & O \\
	\vdots & \vdots &  \ddots & \vdots \\
	O & O&  \dots  & D^{-\frac{1}{2}}\\
	\end{smallmatrix}\right]_{\tiny{p(m+1)}}\cdot
	\left[\begin{smallmatrix}
	A(G) & A(G) &  A(G) & \dots  & A(G)\\
	A(G) & O & O & \dots  & O \\
	\vdots & \vdots & \vdots & \ddots & \vdots \\
	A(G) & O & O & \dots  & O\\
	\end{smallmatrix}\right]_{\tiny{p(m+1)}}\notag
	\left[\begin{smallmatrix}
	((m+1)D)^{-\frac{1}{2}} &O &   \dots  & O\\
	O & D^{-\frac{1}{2}} &  \dots  & O \\
	\vdots & \vdots &  \ddots & \vdots \\
	O & O &  \dots  & D^{-\frac{1}{2}}\\
	\end{smallmatrix}\right]_{\tiny{p(m+1)}}\notag\\
	& = \left[\begin{smallmatrix}
	((m+1)D)^{-\frac{1}{2}}A(G)((m+1)D)^{-\frac{1}{2}}& ((m+1)D)^{-\frac{1}{2}}A(G)D^{-\frac{1}{2}} &  . & \dots  & ((m+1)D)^{-\frac{1}{2}}A(G)D^{-\frac{1}{2}}\\
	((m+1)D)^{-\frac{1}{2}}A(G)D^{-\frac{1}{2}} & O & O & \dots  & O \\
	\vdots & \vdots & \vdots & \ddots & \vdots \\
	((m+1)D)^{-\frac{1}{2}}A(G)D^{-\frac{1}{2}} & O & O & \dots  & O\\
	\end{smallmatrix}\right]_{p(m+1)}\notag\\
	& =	\begin{bmatrix}
	\frac{1}{\sqrt{m+1}} & 1 & 1 & \dots  & 1\\
	1 & 0 & 0 & \dots  & 0 \\
	\vdots & \vdots & \vdots & \ddots & \vdots \\
	1 & 0 & 0 & \dots  & 0\\
	\end{bmatrix}_{m+1}\otimes\frac{1}{\sqrt{m+1}}D^{-\frac{1}{2}}A(G)D^{\frac{-1}{2}}\notag\\&= B\otimes\frac{1}{\sqrt{m+1}}D^{-\frac{1}{2}}A(G)D^{\frac{-1}{2}},\notag
	\textnormal{ where } \quad B =\begin{bmatrix} 
	\frac{1}{\sqrt{m+1}} & 1 & 1 & \dots  & 1\\
	1 & 0 & 0 & \dots  & 0 \\
	\vdots & \vdots & \vdots & \ddots & \vdots \\
	1 & 0 & 0 & \dots  & 0\\
	\end{bmatrix}_{(m+1)}.\notag
	\end{align}	
	The eigenvalues of $B$ are $\sqrt{m+1}, \frac{-m}{\sqrt{m+1}}$ and $0$, and $0$  has  multiplicity $m-1$. So spectrum of $B$ is
	\[Spec(B)=\begin{pmatrix}
	\ 0 &\sqrt{m+1}&\frac{-m}{\sqrt{m+1}}\\
	\ m-1 &1&1\\
	\end{pmatrix}.\]

	Thus the Randi{\'c} spectrum of the $m$-splitting graph is,
{\begin{equation}
		{RS(Spl_m(G))}=\\\left (
		\begin{array}{cccccccccccc}
			0&\rho_1&\rho_2&\dots&\rho_p&\frac{-m}{m+1}\rho_1&\frac{-m}{m+1}\rho_2&\dots&\frac{-m}{m+1}\rho_p\\
		p(m-1)&1&1&\dots&1&1&1&\dots&1
		\end{array}	
		\right ).\notag
		\end{equation}}

	Hence the Randi{\'c} energy of the $m$-splitting graph of $G$ is
$\varepsilon_R(Spl_m(G))=\frac{2m+1}{m+1}\varepsilon_R(G).$

\end{proof}
If $m=1$ in Theorem \ref{3.1}, we get the Randi{\'c} energy of splitting graph of $G$ is $\varepsilon_R(Spl(G))=\frac{3}{2}\varepsilon_R(G)$\cite{chu2019some}.
\begin{cor}

Let $G_1$ and $G_2$ be Randi{\'c} equienergetic graphs. Then $Spl_m(G_1)$ and $Spl_m(G_2)$ are Randi{\'c} equienergetic. 
\end{cor}
In \cite{rojo2012construction}, Rojo et al. have obtained the construction of bipartite graphs having the same Randi{\'c} energy. 
We indicate how to obtain infinitely many pairs of graphs (other than bipartite graphs) having the same Randi{\'c} energy. The following theorem gives some information how to construct a new family of graphs having the same Randi{\'c} energy as that of $G$.
\begin{thm}
Let $G$  be a simple $(p,q)$ graph without isolated vertices. Then Randi{\'c} energy of the $m$-shadow graph of $G$, $m>1$  is $\varepsilon_R(D_m(G))=\varepsilon_R(G).$ 
\end{thm}

\begin{proof}
The Randi{\'c} matrix of the $m$-shadow graph of  $G$  is 
$R(D_m(G))$ \begin{align} &=\left[\begin{smallmatrix}
(mD)^{-\frac{1}{2}} & O &  O & \dots  & O\\
O & (mD)^{-\frac{1}{2}} &O & \dots  & O \\
\vdots & \vdots & \vdots & \ddots & \vdots \\
O & O & O & \dots  & (mD)^{-\frac{1}{2}}\\
\end{smallmatrix}\right]_{pm}\cdot
\left[\begin{smallmatrix}
	A(G) & A(G) &  A(G) & \dots  & A(G)\\
A(G) & A(G) & A(G) & \dots  & A(G) \\
\vdots & \vdots & \vdots & \ddots & \vdots \\
A(G) & A(G) & A(G) & \dots  & A(G)\\
\end{smallmatrix}\right]_{pm}\notag
\left[\begin{smallmatrix}
(mD)^{-\frac{1}{2}} &0 &  O & \dots  & O\\
O & (mD)^{-\frac{1}{2}} & O & \dots  & O \\
\vdots & \vdots & \vdots & \ddots & \vdots \\
O & O & O & \dots  & (mD)^{-\frac{1}{2}}\\
\end{smallmatrix}\right]_{pm}\notag\\
& = \left[\begin{smallmatrix}
(mD)^{-\frac{1}{2}}A(G)(mD)^{-\frac{1}{2}}& (mD)^{-\frac{1}{2}}A(G)(mD)^{-\frac{1}{2}} &  . & \dots  & (mD)^{-\frac{1}{2}}A(G)(mD)^{-\frac{1}{2}}\\
(mD)^{-\frac{1}{2}}A(G)(mD)^{-\frac{1}{2}} & (mD)^{-\frac{1}{2}}A(G)(mD)^{-\frac{1}{2}} & . & \dots  & (mD)^{-\frac{1}{2}}A(G)(mD)^{-\frac{1}{2}} \\
\vdots & \vdots & \vdots & \ddots & \vdots \\
(mD)^{-\frac{1}{2}}A(G)(mD)^{-\frac{1}{2}} & (mD)^{-\frac{1}{2}}A(G)(mD)^{-\frac{1}{2}} & . & \dots  & (mD)^{-\frac{1}{2}}A(G)(mD)^{-\frac{1}{2}}\\
\end{smallmatrix}\right]_{pm}\notag\\
& =\begin{bmatrix}
1 & 1 & 1 & \dots  & 1\\
1 & 1 &1 & \dots  & 1 \\
\vdots & \vdots & \vdots & \ddots & \vdots \\
1 & 1 & 1 & \dots  & 1\\
\end{bmatrix}_{m}\otimes\frac{1}{m}D^{-\frac{1}{2}}A(G)D^{-\frac{1}{2}}\notag\\
& =J_m\otimes\frac{1}{m}D^{-\frac{1}{2}}A(G)D^{-\frac{1}{2}} \notag.
\end{align}
The eigenvalues of $J_m$ are the simple eigenvalue $m$ and $0$, and $0$ has multiplicity $m-1.$
Therefore, the Randi{\'c} spectrum of the $m$-shadow graph is 
 \begin{equation}
	RS(D_m(G))=\left (
	\begin{array}{cccccccc}
	0&\rho_1&\rho_2&\dots&\rho_p\\
	\ p(m-1)&1&1&\dots&1
	\end{array}	
	\right ).\notag
	\end{equation}

Hence $\varepsilon_R(D_m(G))=\varepsilon_R(G).$
\end{proof}
  The following Proposition helps us to construct infinite sequence of Randi{\'c} integral graphs.
\begin{prop}\label{pro1} 
	Let $G$ be a simple $(p,q)$ graph and $m\geq2$ an integer. Then $G$ is Randi{\'c} integral if and only if the $m$-shadow graph of $G$ is Randi{\'c} integral.\\\\
	For example, $D_m(C_4)$ and $D_m(K_{1,4})$ are Randi{\'c} integral for every $m$. 
\end{prop}

The following theorem gives a relation between the Randi{\'c} energy of the $m$-duplicate graph of the graph  and Randi{\'c} energy of the original graph.
\begin{thm}
	Let $G$ be a simple $(p,q)$ graph and $D^m(G)$ be the m-duplicate of graph $G$. Then $\varepsilon_R(D^m(G))=2^m\varepsilon_R(G).$
\end{thm}
\begin{proof}
	The Randi{\'c} matrix of $D^m(G)$ is 
	\begin{align}
	R(D^m(G)) =&\begin{bmatrix}
	O & O &  O & \dots &O & R(G)\\
	O & O & O & \dots  &R(G)& O \\
	\vdots & \vdots & \vdots & \dots & \vdots & \vdots \\
	R(G) & O & O & \dots  & O&O\\
	\end{bmatrix}_{2^mp}\notag\\
	=&\begin{bmatrix}
	0 & 0&  0 & \dots &0 & 1\\
	0 & 0 & 0 & \dots &1 & 0 \\
	\vdots & \vdots & \vdots & \ddots & \vdots&\vdots \\
	1 & 0 & 0 & \dots  & 0&0\\
	\end{bmatrix}_{2^m}\otimes R(G)\notag.
	\end{align}
    Then the Randi{\'c} spectrum of $D^m(G)$ is
	\begin{equation}
	RS(D^m(G))=\left (
	\begin{array}{ccccccccc}
	-\rho_1&-\rho_2&\dots&-\rho_p&\rho_1&\rho_2&\dots&\rho_p\\
	\ 2^{m-1} &2^{m-1}&\dots&2^{m-1}&2^{m-1}&2^{m-1}&\dots&2^{m-1}
	\end{array}	
	\right ).\notag
	\end{equation}
	Hence  Randi{\'c} energy of $D^m(G)$ is $\varepsilon_R (D^m(G))=2^m \varepsilon_R(G)$.
\end{proof} 
\begin{rem}The graphs $D^m(G)$ and $D_{2^m}(G)$ are non-cospectral  equienergetic but not Randi{\'c} equienergetic. 
\end{rem}  
\begin{prop}
	Let $G$ be a simple $(p,q)$ graph and $m\geq 1$. Then $G$ is Randi{\'c} integral if and only if the m-duplicate graph of $G$ is Randi{\'c} integral.
\end{prop}
\section{Energy and Randi{\'c} energy of some non-regular graphs}

In this section, we define some new operations on a graph $G$ and calculate  the energy and Randi{\'c} energy of the resultant graphs.  Moreover, our results allow the construction of new pairs of equienergetic and Randi{\'c} equienergetic graphs.
\begin{op}\label{dh1}
	 Let $G$ be a simple $(p,q)$ graph and  $D_m(G),m>3$ be the the $m$-shadow graph of $G$ and $G_1, G_2,...,G_m$ are the $ m$ copies of $G$ in $D_m(G)$. The graph $H_1^m(G)$ is defined by $H_1^m(G)=D_m(G)-E(G_i)-E(G_j), ~for~a~pair~i\neq j,1\leq i,j\leq m$.
\end{op}
	 The number of vertices and the number of edges in $H_1^m(G)$ are $pm$ and  $(m^2-2)q$ respectively.\\
We can easily compute the energy of $H_1^m(G)$ in terms of energy of $G$.
\begin{thm}
	The energy of the graph $H_1^m(G)$ is $\varepsilon(H_1^m(G)) = \bigg[1+\sqrt{m^2+2m-7}\bigg]\varepsilon(G) $.
\end{thm}
\begin{proof}
	With the suitable labeling of the vertices, the adjacency matrix of $H_1^m(G)$ is 
	\begin{align}
	A(H_1^m(G)) =&\begin{bmatrix}
	O & A(G) &  A(G) & \dots &A(G)& A(G)\\
	A(G) & A(G) & A(G) & \dots  &A(G)& A(G) \\
	\vdots & \vdots & \vdots & \dots & \vdots & \vdots \\
	A(G) & A(G) & A(G) & \dots  &A(G)&A(G)\\
	A(G) & A(G) & A(G) & \dots  &A(G)&O\\
	\end{bmatrix}_{pm}\notag\\
	=&\begin{bmatrix}
	0 & 1 & 1& \dots &1& 1\\
	1 & 1 &1 & \dots  &1& 1 \\
	\vdots & \vdots & \vdots & \dots & \vdots & \vdots \\
	1 & 1 &1 & \dots  &1& 1 \\
	1 & 1 & 1 & \dots  &1&0\\
	\end{bmatrix}_{m}\otimes A(G) =V_1\otimes A(G),\notag\\
    \textnormal{ where } V_1=&\begin{bmatrix}
	0 & 1 & 1& \dots &1& 1\\
	1 & 1 &1 & \dots  &1& 1 \\
	\vdots & \vdots & \vdots & \dots & \vdots & \vdots \\
	1 & 1 &1 & \dots  &1& 1 \\
	1 & 1 & 1 & \dots  &1&0\\
	\end{bmatrix}_m.\notag
	\end{align}
	
	The simple eigenvalues of $V_1$ are  $\frac{m-1+\sqrt{m^2+2m-7}}{2},\frac{m-1-\sqrt{m^2+2m-7}}{2}$,$-1$, and $0$ has multiplicity $m-3$.
	Thus spectrum of $H_1^m(G)$ is\\
	\[Spec(H_1^m(G))=\begin{pmatrix}
	(\frac{m-1+\sqrt{m^2+2m-7}}{2})\lambda_i &(\frac{m-1-\sqrt{m^2+2m-7}}{2})\lambda_i&-\lambda_i&0\\
	1 &1&1&m-3 
	\end{pmatrix},1 \leq i \leq p.\]
 Hence $\varepsilon(H_1^m(G)) = \bigg[1+\sqrt{m^2+2m-7}\bigg]\varepsilon(G) $.
\end{proof}
\begin{cor}
	Let $G_1$ and $G_2$ be equienergetic graphs. Then $H_1^m(G_1)$ and $H_1^m(G_2)$ are equienergetic for all $m>3$.
\end{cor}
\begin{thm}\label{h12}
	The Randi{\'c} energy of the graph $H_1^m(G),m>3$  is $\varepsilon_R(H_1^m(G)) = \varepsilon_R(G)+\frac{2\varepsilon_R(G)}{m} $.
\end{thm}
\begin{proof}
	The Randi{\'c} matrix of  $H_1^m(G)$  is
   $R(H_1^m(G))$ \begin{align}&=\left[\begin{smallmatrix}
	((m-1)D)^{-\frac{1}{2}} & O& \dots  & O\\
	O & (mD)^{-\frac{1}{2}} & \dots  & O \\
	\vdots & \vdots &  \ddots & \vdots \\
	O & O &  \dots  & ((m-1)D)^{-\frac{1}{2}}\\
	\end{smallmatrix}\right]_{pm}\cdot
	\left[\begin{smallmatrix}
	O & A(G) &  \dots  & A(G)\\
	A(G) & A(G) &  \dots  & A(G) \\
	\vdots & \vdots &  \ddots & \vdots \\
	A(G) & A(G) & \dots  & O\\
	\end{smallmatrix}\right]_{pm}\notag
	\left[\begin{smallmatrix}
	((m-1)D)^{-\frac{1}{2}} &O &   \dots  &O\\
	O & (mD)^{-\frac{1}{2}} &  \dots  & O \\
	\vdots & \vdots &  \ddots & \vdots \\
	O & O &  \dots  & ((m-1)D)^{-\frac{1}{2}}\\
	\end{smallmatrix}\right]_{pm}\notag
	\end{align}
	\begin{align*}
	& =	\begin{bmatrix}
0 & \frac{1}{\sqrt{m(m-1)}} & \frac{1}{\sqrt{m(m-1)}} & \dots  & \frac{1}{m-1}\\
\frac{1}{\sqrt{m(m-1)}} & \frac{1}{m} & \frac{1}{m} & \dots  & \frac{1}{\sqrt{m(m-1)}} \\
\vdots & \vdots & \vdots & \ddots & \vdots \\
\frac{1}{m-1} & \frac{1}{\sqrt{m(m-1)}} & \frac{1}{\sqrt{m(m-1)}} & \dots  & 0\\
\end{bmatrix}_{m}\otimes D^{-\frac{1}{2}}A(G)D^{-\frac{1}{2}}\notag\\
&=V_2\otimes D^{-\frac{1}{2}}A(G)D^{-\frac{1}{2}},
\textnormal{ where }  V_2  =\begin{bmatrix}
0 & \frac{1}{\sqrt{m(m-1)}} & \frac{1}{\sqrt{m(m-1)}} & \dots  & \frac{1}{m-1}\\
\frac{1}{\sqrt{m(m-1)}} & \frac{1}{m} & \frac{1}{m} & \dots  & \frac{1}{\sqrt{m(m-1)}} \\
\vdots & \vdots & \vdots & \ddots & \vdots \\
\frac{1}{m-1} & \frac{1}{\sqrt{m(m-1)}} & \frac{1}{\sqrt{m(m-1)}} & \dots  & 0\\
\end{bmatrix}_{m}.
\end{align*}	
The simple eigenvalues of $V_2$ are $1$ , $\frac{-1}{m-1}$ , $\frac{-(m-2)}{m(m-1)}$, and $0$ has multiplicity $m-3$.
	Thus Randi{\'c} spectrum of $H_1^m(G)$ is\\
\[RS(H_1^m(G))=\begin{pmatrix}
\rho_i &\frac{-1}{m-1}\rho_i&\frac{-(m-2)}{m(m-1)}\rho_i&0\\
1 &1&1&m-3 
\end{pmatrix},1 \leq i \leq p.\]
 Thus $\varepsilon_R(H_1^m(G)) = \varepsilon_R(G)+\frac{2\varepsilon_R(G)}{m} $.
\end{proof}
\begin{cor}
	Let $G_1$ and $G_2$ be Randi{\'c} equienergetic graphs. Then $H_1^m(G_1)$ and $H_1^m(G_2)$ are Randi{\'c} equienergetic for all $m>3$.
\end{cor}
\begin{op}\label{dh2}
	Let $G$ be a simple $(p,q)$ graph with vertex set $V(G)=\{v_1,v_2,...,v_p\}$  and $G_1, G_2,...,G_{m-1}$ are the $m-1$ copies of $G,m\geq2$. Define a graph $H_2^m(G)$ with vertex set $V(H_2^m(G))=V(G)\cup \{\cup_{j=1}^{m-1} V(G_j)\}$  and edge set $E(H_2^m(G))$ consisting edges of $G$ and $G_j,1\leq j\leq m-1$ together with those edges joining  $i^{th}$ vertex of $G_j$'s, $1\leq j\leq m-1$, to the neighbors of $v_i$ in $G$, $1\leq i\leq p$.
\end{op}
	The number of vertices and the number of edges in $H_2^m(G)$ are $pm$ and  $(3m-2)q$ respectively.\\    
The following theorem gives a relation between the energy of $H_2^m(G)$ and energy of the original graph.
\begin{thm}\label{4.9}
	The energy of the graph $H_2^m(G)$  is $\varepsilon(H_2^m(G)) = \bigg[m-2+2\sqrt{m-1}\bigg]\varepsilon(G)$, $~m\geq2$.
\end{thm}
\begin{proof}
	With the suitable labeling of the vertices, the adjacency matrix of $H_2^m(G)$ is 
	\begin{align}
	A(H_2^m(G)) =&\begin{bmatrix}
	A(G) & A(G) &  A(G) & \dots &A(G)& A(G)\\
	A(G) & A(G) & O & \dots  &O& O \\
	A(G) &O  & A(G) & \dots  &O& O \\
	\vdots & \vdots & \vdots & \dots & \vdots & \vdots \\
	A(G)& O & O & \dots  &A(G)&O\\
	A(G)& O & O & \dots  &O&A(G)\\
	\end{bmatrix}_{pm}\notag\\
	=&\begin{bmatrix}
	1 & 1 & 1& \dots &1& 1\\
	1 & 1 &0 & \dots  &0& 0 \\
	1 & 0 &1 & \dots  &0& 0 \\
	\vdots & \vdots & \vdots & \dots & \vdots & \vdots \\
	1 & 0 & 0 & \dots  &1&0\\
	1 & 0 & 0 & \dots  &0&1\\
	\end{bmatrix}_{m}\otimes A(G)=W_1\otimes A(G),\notag\\
	\textnormal{where} \quad W_1=&\begin{bmatrix}
	1 & 1 & 1& \dots &1& 1\\
	1 & 1 &0 & \dots  &0& 0 \\
	1 & 0 &1 & \dots  &0& 0 \\
	\vdots & \vdots & \vdots & \dots & \vdots & \vdots \\
	1 & 0 & 0 & \dots  &1&0\\
	1 & 0 & 0 & \dots  &0&1\\
	\end{bmatrix}_m.\notag
	\end{align}
	Let $X=\begin{bmatrix}
	\sqrt{m-1}\\
	1\\
	1\\
	\vdots\\
	1
	\end{bmatrix}_{m \times 1},$
	then $W_1X=(1+\sqrt{m-1})X$ and	let $Y=\begin{bmatrix}
	-\sqrt{m-1}\\
	1\\
	1\\
	\vdots\\
	1
	\end{bmatrix}_{m \times 1},$
	then $W_1Y=(1-\sqrt{m-1})Y$. Let $E_j=\begin{bmatrix}
	0\\
	-1\\
	f_j\\
	\end{bmatrix}_{m\times 1},1 \leq j \leq m-2$, where $f_j$ is the column vector having $j^{th}$ entry one,  all other entries zeros. Then $W_1E_j=E_j.$
	So the  simple eigenvalues of $W_1$ are $1+\sqrt{m-1},1-\sqrt{m-1}$ , and $1$ has multiplicity $m-2$. 
	Thus spectrum of $H_2^m(G)$ is\\
	\[Spec(H_2^m(G))=\begin{pmatrix}
	(1+\sqrt{m-1})\lambda_i &(1-\sqrt{m-1})\lambda_i&\lambda_i\\
	1 &1&m-2
	\end{pmatrix},1 \leq i \leq p.\]
 Hence we get, $\varepsilon(H_2^m(G)) =\bigg[m-2+2\sqrt{m-1}\bigg]\varepsilon(G) $. 
\end{proof}
\begin{cor}
	Let $G$ be an integral graph and $m-1$ a perfect square. Then $H_2^m(G)$ is integral.
\end{cor}
\begin{cor}
	Let $G_1$ and $G_2$ be equienergetic graphs. Then $H_2^m(G_1)$ and $H_2^m(G_2)$ are equienergetic for all $m>1$.
\end{cor}
\begin{rem}
	If $m=2$, the graph $H_2^m(G)$ coincide with the shadow graph $D_2(G)$.
\end{rem}
\begin{thm}\label{h22}
	The Randi{\'c} energy of the graph $H_2^m(G),m>3$ is $\varepsilon_R(H_2^m(G)) = \varepsilon_R(G)+\frac{(m+1)(m-2)\varepsilon_R(G)}{2m} $.
\end{thm}
\begin{proof}
	The Randi{\'c} matrix of $H_2^m(G)$  is $R(H_2^m(G))$
	\begin{align}  &=\left[\begin{smallmatrix}
	(mD)^{-\frac{1}{2}} & O &  O & \dots  & O\\
	O & (2D)^{-\frac{1}{2}} & O & \dots  & O \\
	\vdots & \vdots & \vdots & \ddots & \vdots \\
	O & O & O & \dots  & (2D)^{-\frac{1}{2}}\\
	\end{smallmatrix}\right]_{pm}\cdot
	\left[\begin{smallmatrix}
    A(G) & A(G) &  A(G) & \dots  & A(G)\\
	A(G) & A(G) &O & \dots  & O \\
	\vdots & \vdots & \vdots & \ddots & \vdots \\
	A(G) & O &O & \dots  & A(G)\\
	\end{smallmatrix}\right]_{pm}\notag
	\left[\begin{smallmatrix}
	(mD)^{-\frac{1}{2}} &0 & O & \dots  & O\\
	O & (2D)^{-\frac{1}{2}} & O & \dots  & O \\
	\vdots & \vdots & \vdots & \ddots & \vdots \\
	O & O & O & \dots  & (2D)^{-\frac{1}{2}}\\
	\end{smallmatrix}\right]_{pm}\notag
	\end{align}
	\begin{align}
	&=\begin{bmatrix}
	\frac{1}{m} & \frac{1}{\sqrt{2m}} & \frac{1}{\sqrt{2m}} & \dots  & \frac{1}{\sqrt{2m}}\\
	\frac{1}{\sqrt{2m}} & \frac{1}{2} & 0 & \dots  & 0 \\
	\vdots & \vdots & \vdots & \ddots & \vdots \\
	\frac{1}{\sqrt{2m}} & 0 & 0 & \dots  & \frac{1}{2}\\
	\end{bmatrix}_{m}\otimes D^{-\frac{1}{2}}A(G)D^{-\frac{1}{2}}\notag\\
	&=W_2\otimes D^{-\frac{1}{2}}A(G)D^{-\frac{1}{2}},
	\textnormal{ where } \quad W_2  =\begin{bmatrix}
	\frac{1}{m} & \frac{1}{\sqrt{2m}} & \frac{1}{\sqrt{2m}} & \dots  & \frac{1}{\sqrt{2m}}\\
	\frac{1}{\sqrt{2m}} & \frac{1}{2} & 0 & \dots  & 0 \\
	\vdots & \vdots & \vdots & \ddots & \vdots \\
	\frac{1}{\sqrt{2m}} & 0 & 0 & \dots  & \frac{1}{2}\\
	\end{bmatrix}_{m}.\notag
	\end{align}	
	Let $X^*=\begin{bmatrix}
\sqrt{\frac{m}{2}}\\
1\\
1\\
\vdots\\
1
\end{bmatrix}_{m\times 1},$
then $W_2X^*=1.X^*$ and	let $Y^*=\begin{bmatrix}
-(m-1)\sqrt{\frac{2}{m}}\\
1\\
1\\
\vdots\\
1
\end{bmatrix}_{m\times 1},$
then $W_2Y^*=\frac{-(m-2)}{2m}Y^*$. Let $E_j$ be as in Theorem \ref{4.9},  then $W_2E_j=\frac{1}{2}E_j.$
So the simple eigenvalues of $W_2$ are $1,\frac{-(m-2)}{2m}$, and $\frac{1}{2}$ has multiplicity $m-2$.  
Thus Randi{\'c} spectrum of $H_2^m(G)$ is\\
\[RS(H_2^m(G))=\begin{pmatrix}
\rho_i &\frac{-(m-2)}{2m}\rho_i&\frac{1}{2}\rho_i\\
1 &1&m-2 
\end{pmatrix},1 \leq i \leq p.\]
  Hence $\varepsilon_R(H_2^m(G)) = \varepsilon_R(G)+\frac{(m+1)(m-2)\varepsilon_R(G)}{2m} $.
\end{proof}
\begin{cor}
	Let $G_1$ and $G_2$ be Randi{\'c} equienergetic graphs. Then $H_2^m(G_1)$ and $H_2^m(G_2)$ are Randi{\'c} equienergetic for all $m>3$.
\end{cor}
\begin{op}\label{dh5}
	Let $G$ be a simple $(p,q)$ graph with vertex set $V(G)=\{v_1,v_2,...,v_p\}$  and $G_1, G_2,...,G_{m-1}$ are the $m-1$ copies of $G$.  Define a graph $H_3^m(G),m>1$ with vertex set $V(H_3^m(G))=V(G)\cup \{\cup_{i=1}^{m-1} V(G_i)\}$  and edge set $E(H_3^m(G))$ consisting only of those edges joining  $i^{th}$ vertex of $G_j$ $1\leq j\leq m-1$, to the neighbors of $v_i$ in $G$, $1\leq i\leq p$ and then removing edges of $G$.
\end{op}
	The number of vertices and the number of edges in $H_3^m(G)$ are $pm$ and $3(m-1)q$ respectively.

\begin{thm}\label{h51}
	The energy of the graph $H_3^m(G)$ is $\varepsilon(H_3^m(G)) = \bigg[m-2+\sqrt{4m-3}\bigg]\varepsilon(G) $.
\end{thm}
\begin{proof}
	The adjacency matrix of $H_3^m(G)$ is 
	\begin{align}
	A(H_3^m(G)) =&\begin{bmatrix}
	O & A(G) &  A(G) & \dots &A(G)& A(G)\\
	A(G) & A(G) & O & \dots  &O& O \\
	\vdots & \vdots & \vdots & \dots & \vdots & \vdots \\
	A(G)& O & O & \dots  &O&A(G)\\
	\end{bmatrix}_{pm}\notag\\
	=&\begin{bmatrix}
	0 & 1 & 1& \dots &1& 1\\
	1 & 1 &0 & \dots  &0& 0 \\
	\vdots & \vdots & \vdots & \dots & \vdots & \vdots \\
	1 & 0 & 0 & \dots  &0&1\\
	\end{bmatrix}_{m}\otimes A(G)\notag\\
	=&Z_1\otimes A(G),
	\textnormal{ where } Z_1=\begin{bmatrix}
	0 & 1 & 1& \dots &1& 1\\
	1 & 1 &0 & \dots  &0& 0 \\
	\vdots & \vdots & \vdots & \dots & \vdots & \vdots \\
	1 & 0 & 0 & \dots  &0&1\notag\\
	\end{bmatrix}.
	\end{align}
	Let $P=\begin{bmatrix}
	\frac{-1+\sqrt{4m-3}}{2}\\
	1\\
	1\\
	\vdots\\
	1
	\end{bmatrix}_{m\times1}$,
	then $Z_1P=\big(\frac{1+\sqrt{4m-3}}{2}\big)P$ and	let $Q=\begin{bmatrix}
	\frac{-1-\sqrt{4m-3}}{2}\\
	1\\
	1\\
	\vdots\\
	1
	\end{bmatrix}_{m\times1}$,
	then $Z_1Q=\big(\frac{1-\sqrt{4m-3}}{2}\big)Q.$ Let $E_j$ be as in Theorem \ref{4.9}, then $Z_1E_j=E_j.$
	So the simple eigenvalues of $Z_1$ are $\frac{1+\sqrt{4m-3}}{2},\frac{1-\sqrt{4m-3}}{2}$, and $ 1$ has multiplicity $m-2$.  
	 Thus spectrum of $H_3^m(G)$ is\\
	 \[Spec(H_3^m(G))=\begin{pmatrix}
	 (\frac{1+\sqrt{4m-3}}{2})\lambda_i &(\frac{1-\sqrt{4m-3}}{2})\lambda_i&\lambda_i\\
	 1 &1&m-2
	 \end{pmatrix},1 \leq i \leq p.\]
	 Hence we have \\ $\varepsilon(H_3^m(G)) =\bigg[m-2+\sqrt{4m-3}\bigg]\varepsilon(G) $. 
\end{proof}
\begin{cor}
Let $G$ be an integral graph. Then $H_3^m(G)$ is an integral graph if $4m-3$ is a perfect square. 
\end{cor}
For example, $H_3^3(K_2)$, $H_3^7(K_2)$, $H_3^{13}(K_2)$, $H_3^{21}(K_2)$ etc.
\begin{cor}
	Let $G_1$ and $G_2$ be equienergetic graphs. Then $H_3^m(G_1)$ and $H_3^m(G_2)$ are equienergetic for all $m>1$.
\end{cor}
\begin{thm}\label{4.17}
	The Randi{\'c} energy of graph $H_3^m(G),m>2$ is $\varepsilon_R(H_3^m(G)) = \varepsilon_R(G)+\frac{(m-1)\varepsilon_R(G)}{2} $.
\end{thm}
\begin{proof}
	The Randi{\'c} matrix of $H_3^m(G)$  is $R(H_3^m(G))$
	\begin{align*} 
	=&\left[\begin{smallmatrix}
	((m-1)D)^{-\frac{1}{2}} & O &  O & \dots  & O\\
	O & (2D)^{-\frac{1}{2}} & O & \dots  & O \\
	\vdots & \vdots & \vdots & \ddots & \vdots \\
	O & O & O & \dots  & (2D)^{-\frac{1}{2}}\\
	\end{smallmatrix}\right]_{pm}\cdot
	\left[\begin{smallmatrix}
    O & A(G) &  A(G) & \dots  & A(G)\\
	A(G) & A(G) &O & \dots  & O \\
	\vdots & \vdots & \vdots & \ddots & \vdots \\
	A(G) & O &O & \dots  & A(G)\\
	\end{smallmatrix}\right]_{pm}\notag
	\left[\begin{smallmatrix}
	((m-1)D)^{-\frac{1}{2}} &O &  O & \dots  & O\\
	O & (2D)^{-\frac{1}{2}} & O& \dots  & O \\
	\vdots & \vdots & \vdots & \ddots & \vdots \\
	O & O & O & \dots  & (2D)^{-\frac{1}{2}}\\
	\end{smallmatrix}\right]_{pm}\notag\\
	\end{align*}
	\begin{align}
	 &=	\begin{bmatrix}
	0 & \frac{1}{\sqrt{2(m-1)}} & \frac{1}{\sqrt{2(m-1)}} & \dots  & \frac{1}{\sqrt{2(m-1)}}\\
	\frac{1}{\sqrt{2(m-1)}} & \frac{1}{2} & 0 & \dots  & 0 \\
	\vdots & \vdots & \vdots & \ddots & \vdots \\
	\frac{1}{\sqrt{2(m-1)}} & 0 & 0 & \dots  & \frac{1}{2}\\
	\end{bmatrix}_{m}\otimes D^{\frac{-1}{2}}A(G)D^{\frac{-1}{2}}\notag\label{eq4}\\
	&=Z_2\otimes D^{\frac{-1}{2}}A(G)D^{\frac{-1}{2}},\textnormal{ where }\notag\\
	 Z_2 & =\begin{bmatrix}
	0 & \frac{1}{\sqrt{2(m-1)}} & \frac{1}{\sqrt{2(m-1)}} & \dots  & \frac{1}{\sqrt{2(m-1)}}\\
	\frac{1}{\sqrt{2(m-1)}} & \frac{1}{2} & 0 & \dots  & 0 \\
	\vdots & \vdots & \vdots & \ddots & \vdots \\
	\frac{1}{\sqrt{2(m-1)}} & 0 & 0 & \dots  & \frac{1}{2}\\
	\end{bmatrix}_{m}.\notag
    \end{align}	
	Let $P^*=\begin{bmatrix}
	\sqrt{\frac{m-1}{2}}\\
	1\\
	1\\
	\vdots\\
	1
	\end{bmatrix}_{m\times1}$
	then $Z_2P^*=1.P^*$ and	let $Q^*=\begin{bmatrix}
	-(m-1)\sqrt{\frac{2}{m-1}}\\
	1\\
	1\\
	\vdots\\
	1
	\end{bmatrix}_{m\times1}$
	then $Z_2Q^*=-\frac{1}{2}Q^*$. Let $E_j$ be as in Theorem \ref{4.9}, then $Z_2E_j=\frac{1}{2}E_j$. 
	 So the simple eigenvalues of $Z_2$ are $1,-\frac{1}{2}$, and $\frac{1}{2}$ has multiplicity $m-2$. 
	 Thus Randi{\'c} spectrum of $H_3^m(G)$ is\\
	 \[RS(H_3^m(G))=\begin{pmatrix}
	 \rho_i &-\frac{1}{2}\rho_i&\frac{1}{2}\rho_i\\
	 1 &1&m-2 
	 \end{pmatrix},1 \leq i \leq p.\]
	    Hence $\varepsilon_R(H_3^m(G)) = \varepsilon_R(G)+\frac{(m-1)\varepsilon_R(G)}{2} $.
\end{proof}
\begin{cor}
	Let $G_1$ and $G_2$ be Randi{\'c} equienergetic graphs. Then $H_3^m(G_1)$ and $H_3^m(G_2)$ are Randi{\'c} equienergetic for all $m>1$.
\end{cor}
\section{Applications}
In this section, we compute the degree Kirchhoff index,  the Kemeny's constant and the number of spanning trees of $Spl_m(G)$ in terms of original graph. Analogous to this, results for $D_m(G)$, $D^m(G)$,$H_1^m(G)$, $H_2^m(G)$ and $H_3^m(G)$ are included in Appendix.
\begin{thm}
	Let $G$ be a simple connected $(p,q)$ graph with Randi{\'c} spectrum  \{$\rho_1,\rho_2,...,\rho_p$\}. Then \begin{align*}
		 K(Spl_m(G))=p(m-1)+K(G)+\sum_{i=2}^{p}\frac{m+1}{1+m(1+\rho_i(G))}.\\
		\end{align*}
\end{thm}
\begin{thm}
	Let $G$ be a simple connected $(p,q)$ graph with Randi{\'c} spectrum  \{$\rho_1,\rho_2,...,\rho_p$\}. Then
	 \begin{align*}
     Kf^*(Spl_m(G))=2(m+1)q\bigg[p(m-1)+\sum_{i=2}^{p}\frac{m+1}{1+m(1+\rho_i(G))}\bigg]+(m+1)Kf^*(G).
	\end{align*}

\end{thm}
From the following theorem, we obtain the number of spanning trees of graphs in terms of Randi{\'c} eigenvalues.
\begin{thm}
Let $G$ be a connected simple $(p,q)$ graph with Randi{\'c} spectrum  \{$\rho_1,\rho_2,...,\rho_p$\}. Then \begin{align*}
	t(Spl_m(G))=\frac{(m+1)^{p}(\prod_{i=1}^{p}d_i)^{m}t(G)\prod_{i=1}^{p}
	\left(1+\frac{m\rho_i(G)}{m+1}\right)}{2m+1}.\\
\end{align*}
\end{thm}
\subparagraph{Conclusion}
In this paper, we compute the energy and Randi{\'c} energy of some specific graphs which are obtained by some graph operations on $G$. Also, our results show how to construct some new class of graphs having the same Randi{\'c} energy as that of $G$. In addition, some new family of  equienergetic, Randi{\'c} equienergetic, integral and Randi{\'c} integral graphs are obtained. Moreover, we discuss some graph invariants like the degree Kirchhoff index, the Kemeny's constant and the number of spanning trees of graph $Spl_m(G)$.
\section{Appendix}
Let $G$ be a connected graph, then $D_m(G)$, $H_1^m(G)$, $H_2^m(G)$ and $H_3^m(G)$ are connected. Here we discuss the degree Kirchhoff index, the Kemeny's constant and the number of spanning trees of graphs $D_m(G)$, $H_1^m(G)$, $H_2^m(G)$ and $H_3^m(G)$.
\begin{thm}
	Let $G$ be a simple connected $(p,q)$ graph with Randi{\'c} spectrum  \{$\rho_1,\rho_2,...,\rho_p$\}. Then \begin{enumerate}
		\item $K(D_m(G))=p(m-1)+K(G)$.\\
		\item $K(H_1^m(G))=m-3+K(G)+\sum_{i=2}^{p}\frac{m-1}{m-1+\rho_i(G)}+\sum_{i=2}^{p}\frac{m(m-1)}{m^2-m+(m-2)\rho_i(G)}$.\\
		\item $K(H_2^m(G))=m-3+K(G)+\sum_{i=2}^{p}\frac{2m}{2m+(m-2)\rho_i(G)}+\sum_{i=2}^{p}\frac{2}{2-\rho_i(G)}$.\\
		\item $K(H_3^m(G))=6(m-1)q\left[K(G)+\sum_{i=2}^{p}\frac{2}{2+\rho_i(G)}+\sum_{i=2}^{p}\frac{2}{2-\rho_i(G)}\right].$
	\end{enumerate}
\end{thm}
\begin{thm}
	Let $G$ be a simple connected $(p,q)$ graph with Randi{\'c} spectrum  \{$\rho_1,\rho_2,...,\rho_p$\}. Then \begin{enumerate}
		\item $Kf^*(D_m(G))=2m^2q\left(p(m-1)\right)+m^2Kf^*(G).$\\
		\item $Kf^*(H_1^m(G))=2(m^2-2)q\bigg[m-3+\sum_{i=2}^{p}\frac{m-1}{m-1+\rho_i(G)}+\sum_{i=2}^{p}\frac{m(m-1)}{m^2-m+(m-2)\rho_i(G)}\bigg]+(m^2-2)Kf^*(G)$.\\
		\item $Kf^*(H_2^m(G))=2(3m-2)q\left[m-3+\sum_{i=2}^{p}\frac{2m}{2m+(m-2)\rho_i(G)}+\sum_{i=2}^{p}\frac{2}{2-\rho_i(G)}\right]+(3m-2)Kf^*(G)$.\\
		\item $Kf^*(H_3^m(G))=6(m-1)q\left[\sum_{i=2}^{p}\frac{2}{2+\rho_i(G)}+\sum_{i=2}^{p}\frac{2}{2-\rho_i(G)}\right]+3(m-1)Kf^*(G).$
	\end{enumerate}
\end{thm}
From the following theorem, we obtain the number of spanning trees of graphs in terms of Randi{\'c} eigenvalues.
\begin{thm}
	Let $G$ be a connected simple $(p,q)$ graph with Randi{\'c} spectrum  \{$\rho_1,\rho_2,...,\rho_p$\}. Then \begin{enumerate}
		\item $t(D_m(G))=\frac{m^{mp}(\prod_{i=1}^{p}d_i)^{m-1}t(G)}{m^2}.$\\
		\item $t(H_1^m(G))=\frac{m^{(m-2)p}(m-1)^2(\prod_{i=1}^{p}d_i)^{m-1}t(G)\prod_{i=1}^{p}
			\left(1+\frac{\rho_i(G)}{m-1}\right)\prod_{i=1}^{p}
			\left(1+\frac{(m-2)\rho_i(G)}{m(m-1)}\right)}{m^2-2}.$\\
		\item $t(H_2^m(G))=\frac{2^{(m-1)p}m^p(\prod_{i=1}^{p}d_i)^{m-1}t(G)\prod_{i=1}^{p}
			\left(1+\frac{(m-2)\rho_i(G)}{2m}\right)\prod_{i=1}^{p}
			\left(1+\frac{\rho_i(G)}{2}\right)}{3m-2}.$\\
		\item $t(H_3^m(G))=\frac{2^{(m-1)p}(m-1)^p(\prod_{i=1}^{p}d_i)^{m-1}t(G)\prod_{i=1}^{p}
			\left(1-\frac{\rho_i(G)}{2}\right)\prod_{i=1}^{p}
			\left(1+\frac{\rho_i(G)}{2}\right)}{3m-3}.$
	\end{enumerate}
\end{thm}
\bibliography{refe1}
\bibliographystyle{plain}

\end{document}